\documentclass[11pt]{article}

	\usepackage{mathtools}
	\usepackage{amssymb}
	\usepackage{amsthm}
	\usepackage{amsmath}
	\usepackage{mathrsfs}
	\usepackage{tikz}
	\usepackage{tikz-cd}
		\usetikzlibrary{decorations.pathmorphing}
	\usepackage{xcolor} 
	\usepackage{stmaryrd}
	\usepackage{bbm}
	\usepackage{adjustbox}	
	\usepackage{setspace}
	\usepackage{geometry}
	\usepackage{array} 
		\newlength\mylen
		\newcolumntype{C}{>{\hfil$}p{\mylen}<{$\hfil}}
	\usepackage{graphicx}
		\graphicspath{ {./Images/} }
	\usepackage{caption}
	\usepackage{subcaption}

\newcommand{\pair}[1]{\left\langle #1\right\rangle}
\newcommand{\tate}[1]{\left\langle #1\right\rangle}
\newcommand{\ldef}[1]{\textcolor{purple}{\textit{#1}}}

\DeclareMathOperator{\sign}{sign}

\DeclareMathOperator{\Tilt}{Tilt}
\DeclareMathOperator{\Proj}{Proj}
\DeclareMathOperator{\Inj}{Inj}
\DeclareMathOperator{\Perv}{Perv}
\DeclareMathOperator{\mix}{mix}
\DeclareMathOperator{\Bim}{Bim}
\DeclareMathOperator{\Mod}{Mod}
\DeclareMathOperator{\SBim}{\mathbb{S}Bim}
\DeclareMathOperator{\BSBim}{\mathbb{BS}Bim}
\DeclareMathOperator{\SMod}{\mathbb{S}Mod}

\DeclareMathOperator{\R}{\mathbb{R}}
\DeclareMathOperator{\C}{\mathbb{C}}
\DeclareMathOperator{\Q}{\mathbb{Q}}
\DeclareMathOperator{\Z}{\mathbb{Z}}

\DeclareMathOperator{\bV}{\mathbb{V}}

\DeclareMathOperator{\h}{\mathfrak{h}}

\DeclareMathOperator{\tA}{\textbf{A}}
\DeclareMathOperator{\tB}{\textbf{B}}
\DeclareMathOperator{\tC}{\textbf{C}}
\DeclareMathOperator{\tD}{\textbf{D}}
\DeclareMathOperator{\tE}{\textbf{E}}
\DeclareMathOperator{\tF}{\textbf{F}}
\DeclareMathOperator{\tH}{\textbf{H}}
\DeclareMathOperator{\tI}{\textbf{I}}
\DeclareMathOperator{\tR}{\textbf{R}}

\DeclareMathOperator{\Hom}{Hom}

\DeclareMathOperator{\grrk}{rk^{\bullet}}
\DeclareMathOperator{\id}{id}

\DeclareMathOperator{\Sym}{Sym}
\DeclareMathOperator{\alphac}{\check{\alpha}}

\newtheorem{thm}{Theorem}

\newtheorem{lem}[thm]{Lemma}
\newtheorem{prop}[thm]{Proposition}

\theoremstyle{remark}
\newtheorem{rem}[thm]{Remark}
\newtheorem{exmp}[thm]{Example}

\title{On the coefficients in the Jones-Wenzl idempotent}
\author{J. Baine}
\date{\emph{In memory of Greg Wiley}}
\begin{document}

	\maketitle
	\begin{abstract}
		\noindent
		By studying a categorification of the antisymmetriser quasi-idempotent in the Hecke algebra, we derive a closed formula for the Jones-Wenzl idempotent in the Temperley-Lieb algebra. 
		In particular, we show that when the idempotent is expressed in terms of the monomial basis, the coefficients are the graded ranks of certain indecomposable Soergel modules.
		Equivalently, the coefficients can be expressed as a ratio of certain Kazhdan-Lusztig polynomials. 
		Similar results are obtained for generalised Jones-Wenzl idempotents in other types. 
	\end{abstract}
	
	\noindent
	The Jones-Wenzl idempotent $j_n$ is an element of the Temperley-Lieb algebra  which projects onto the trivial submodule. 
	Its importance stems from its ubiquity in mathematics: 
	in knot theory it is used to compute the coloured Jones polynomial of a knot; 
	in representation theory it arises in the endomorphism algebra of tensor powers of the natural representations of $SL_2(\C)$ and $U_{q}(\mathfrak{sl}_2)$; and, 
	in Soergel bimodule theory it appears in the defining relations of the Bott-Samelson category. 
	\\
	\par 
	Given its ubiquity, a natural question first posed by V. Jones is to determine a formula for $j_n$. 
	Wenzl famously determined a recursive relation in \cite{Wen87}, with a second recursive relation being determined in \cite{FK97}. 
	However closed formulas remained elusive; indeed, Ocneanu writes ``The general opinion among mathematicians and physicists, who had been searching for such a formula for applications in quantum field theory, appeared to be that such a closed formula might not exist in general.''  
	Ocneanu announced a closed formula in \cite{Ocn02}, which has been proven in very limited cases \cite{Rez07}, and Morrison determined an algorithm to compute coefficients in \cite{Mor17}. 
	\\
	\par 
	The main result of this note is the following non-recursive formula for $j_n$ in terms of Kazhdan-Lusztig polynomials. 
	Moreover, the coefficients arise naturally as ratios of the graded ranks of certain indecomposable Soergel modules.  
	
	\begin{thm}
	\label{Thm: JW type A}
		Let $W_{FC}$ be the set of fully-commutative elements in Type $\tA_{n-1}$, and $u_x \in TL_n$ the corresponding monomial (i.e. diagrammatic) basis element.
		For any $x \in W_{FC}$, the coefficient of $u_x$ in $j_n$ is 
		\begin{align*}
			\frac{(-1)^{\ell(x)}}{[n]!} \sum_{y} v^{-\ell(y)} h_{y,xw_0}
		\end{align*}
		where $h_{y, xw_0}$ is a Kazhdan-Lusztig polynomial and $[n]!$ is the $n$-th quantum factorial. 
		Equivalently, $j_n$ admits an expression in terms of the graded ranks of Soergel modules as 
		\begin{align*}
			j_n = 
			\sum_{x \in W_{FC}} 
			(-1)^{\ell(x)}
			~\frac{\grrk \check{B}_{ x w_0}}{\grrk \check{B}_{w_0}}
			~u_x.
		\end{align*}  
	\end{thm}
	
	More generally, in Theorem \ref{Thm: gen JW} we show that similar identities hold for the generalised Jones-Wenzl idempotents considered in \cite{Sen19}.
	\\
	\par 
	Obviously, no closed formula is currently known for Kazhdan-Lusztig polynomials.
	So, the reader may argue that Theorem \ref{Thm: JW type A} is not closed.
	Alternatively, this formula is the simplest we know; it explains the difficulties encountered by others in finding a closed form; and, it elucidates that these coefficients are of deep mathematical significance. 
	\\
	\par 
	We begin by introducing various Hecke categories and their associated functors in sections \ref{Sec: realisations} - \ref{Sec: functors}. 
	We then commence a study of Jordan-H\"{o}lder multiplicities of the big tilting object $T_{w_0}$ in sections \ref{Sec: ranks} - \ref{Sec:JH mult}.
	In section \ref{Sec: Hecke algebras} we show that $T_{w_0}$ categorifies the antisymmetriser quasi-idempotent in the Hecke algebra. 
	Finally, we exploit the properties of IC-bases of (generalised) Temperley-Lieb algebras to deduce the formulas for (generalised) Jones-Wenzl idempotents in sections \ref{Sec: TL} and \ref{Sec: generalised}.

\section{Realisations of Coxeter systems}
\label{Sec: realisations}
	
	We will consider various categories whose split Grothendieck groups are isomorphic to Hecke algebras. 
	The data required to construct these categories is a realisation, in the sense of \cite{EW16}. 
	The realisations we consider are: the root realisations of Weyl groups, so that we have recourse to geometry; and, geometric realisations so that we can extend our results to all finite Coxeter groups.
	\\
	\par 
	Let $(W,S)$ be a Coxeter system with length function $\ell$ and Bruhat order $\leq$.
	Further assume $W$ is finite with longest element $w_0$. 
	For any simple reflections $s,t \in S$ let $m_{st} \in \Z_{\geq 0}$ denote the order of $st$.
	\\
	\par 
	A \ldef{realisation} of a Coxeter system $(W,S)$ is a triple $(\h , \{ \alphac_s \} , \{ \alpha_s \})$ consisting of: a finite-rank, free $\Bbbk$-module $\h$; a collection of elements $\{ \alphac_s \} \subset \h$; the dual space $\h^{*} := \Hom_{\Bbbk}(\h, \Bbbk)$; and, a collection of elements $\{ \alpha_s \} \subset \h^*$, all of which satisfy: $\pair{\alphac_s, \alpha_s} =2 $ for all $s \in S$; the assignment $s(\lambda) = \lambda - \pair{\lambda, \alpha_s}\alphac_s$ for each $s \in S$ and $\lambda \in \h$ defines a $W$-module structure on $\h$; and, a technical condition discussed in \cite[\S 3.1]{EW16} and \cite[\S 2.1]{RV23}.
	We will always take $\Bbbk$ to be a field of characteristic $0$. 
	\\
	\par 
	A realisation is said to be \ldef{reflection faithful} if $W$ acts faithfully on $\h$, and $\text{codim} \h^w =1$ if and only if $w$ is a reflection, i.e. conjugate to $s \in S$. 	\\
	\par 
	If $(W,S)$ is an arbitrary, finite Coxeter system, we define $\h := \bigoplus_{s \in S} \R \alphac_s$, and elements $\{ \alpha_s \vert s \in S \} \subset \h^{*}$ are chosen so that 
	\begin{align*}
		\pair{\alphac_s , \alpha_t} = -2 \cos (\pi / m_{st})
	\end{align*} 
	where $\pair{-,-} : \h \times \h^* \rightarrow \R$ is the natural pairing. 
	The data $(\h, \{ \alphac_s \} , \{ \alpha_s \} )$ is called the \ldef{geometric realisation} of $(W,S)$.
	This realisation is reflection faithful by \cite{Soe07}. 
	\\
	\par 
	If $(X, R, \check{X}, \check{R})$ is the reduced root datum of an algebraic group $G$. 
	Fix a set of simple roots $\{ \alpha_s \} \subset R$, so that the associated simple reflections $S$ in the Weyl group $W$ endow $(W,S)$ with the structure of a Coxeter system. 
	The data $(X, \{ \alphac_s \} , \{ \alpha_s \} ) $ is a realisation of $(W,S)$ over $\Z$. 
	Fix $\Bbbk$, a field of characteristic 0, and set $\h := X \otimes \Bbbk$. 
	Then, $(\h, \{ \alphac_s \} , \{ \alpha_s \} ) $ is a realisation over $\Bbbk$ called the \ldef{root realisation} of the Coxeter system $(W,S)$.
	This realisation is reflection faithful by \cite{Soe90, Soe92}. 
	\\
	\par
	Given a realisation $(\h, \{ \alphac_s \} , \{ \alpha_s \} )$, the data $(\h^*, \{ \alpha_s \} , \{ \alphac_s \} )$ is also a realisation of $(W,S)$ called the \ldef{Langlands dual realisation} of $(W,S)$.
	This nomenclature stems from the fact that the Langlands dual realisation of the root realisation associated to an algebraic group $G$ is the root realisation associated to the Langlands dual algebraic group $\check{G}$.
	\\
	\par 
	Where it will not cause any confusion we abbreviate the data of a realisation $(\h, \{ \alphac_s \} , \{ \alpha \})$ to $\h$, and write $\check{\h}$ for the Langlands dual realisation. 
	
\section{Mixed perverse Hecke categories}
\label{Sec: categories}
	
	In this section we introduce various categories associated to a reflection faithful realisation $\h$. 
	Ultimately, our attention will be focused on indecomposable Soergel modules and indecomposable tilting complexes.
	\\
	\par 
	Fix a Coxeter system $(W,S)$ where $W$ is finite, and a reflection faithful realisation $\h$ defined over a field $\Bbbk$. 
	The reader is welcome to only consider the case of $\Bbbk=\R$, however these constructions hold more generally. 
	\\
	\par
	Denote by $R:= \Sym(\h)$ the symmetric algebra on the $\Bbbk$-vector-space $\h$, where $\alpha \in \h$ is considered as an element in degree 2.
	Let $R$-$\Bim$ denote the category of $\Z$-graded \ldef{$R$-bimodules}. 
	It is endowed with a shift functor $(1)$, and tensoring over $R$, i.e. $- \otimes_R -$, induces a monoidal structure on $R$-$\Bim$. 
	Now consider $R$ as a $\Z$-graded $W$-module by extending the $W$-module structure on $\h$. 
	For any $s \in S$ let $R^s$ denote the subring of invariants under $s$, and define $B_s := R \otimes_{R^s} R (1)$.
	It is easy to check $B_s$ is in $R$-$\Bim$, and $B_s$ is free as a left or right $R$-module. 
	\\
	\par 
	The category of \ldef{Bott-Samelson bimodules} $\BSBim(\h)$ is the full, monoidal subcategory of $R$-$\Bim$ which is monoidally generated by $R$ and $B_s$ for each $s \in S$. 
	Any bimodule in $\BSBim(\h)$ is called a Bott-Samelson bimodule. 
	\\
	\par 	
	The category of \ldef{Soergel bimodules} $\SBim(\h)$ is the  strictly-full subcategory of $R$-$\Bim$ whose objects are finite direct-sums of direct-summands of Bott-Samelson bimodules, and their shifts. 
	It is a Krull-Schmidt, additive, monoidal category with shift functor $(1)$. 
	For any $B,B'$ in $\SBim(\h)$ we define 
	\begin{align*}
		\Hom_{\SBim(\h)}^{\bullet} (B,B') 
		:= 
		\bigoplus_{k \in \Z} \Hom_{\SBim(\h)}(B,B'(k)) 
	\end{align*}
	which is a graded $R$-bimodule, and is free as a graded, left $R$-module \cite[\S5]{Soe07} . 
	For each $x \in W$ and $n \in \Z$ there is an indecomposable object $B_x(n)$, and each indecomposable object is isomorphic to an object of this form. 
	Moreover, if $s_1 \dots s_k$ is a reduced expression for $x$, i.e. $x= s_1 \dots s_k$ and $\ell(x)=k$, then the $B_x$ occurs as a direct summand of $B_{s_1} \otimes_R \dots \otimes_R B_{s_k}$ with multiplicity 1, and does not occur as a summand of any Bott-Samelson bimodule with fewer tensor factors \cite[\S6]{Soe07}. 
	\\
	\par 	
	Let $\Mod$-$R$ denote the category of graded, right $R$-modules. 
	Endowing $\Bbbk$ with a trivial $R$-module structure induces a functor $\Bbbk \otimes_R (-) : R\text{-}\Bim \rightarrow \Mod\text{-}R$. 
	The category of right \ldef{Soergel modules} $\SMod(\h)$ is the essential image of the restriction of $\Bbbk \otimes_R (-) $ to $\SBim(\h)$. 
	It is a Krull-Schmidt, additive category with shift functor $(1)$. 
	For any $B,B'$ in $\SMod(\h)$ we define 
	\begin{align*}
		\Hom_{\SMod(\h)}^{\bullet} (B,B') 
		:= 
		\bigoplus_{k \in \Z} \Hom_{\SMod(\h)}(B,B'(k)). 
	\end{align*}
	which is free as a graded, left $\Bbbk$-module \cite[\S 1.7]{Ric19}. 
	For each $x \in W$ and $n \in \Z$ the module $\Bbbk \otimes_{R} B_x(n)$ is indecomposable \cite[\S1.7]{Ric19}.  
	We will abuse notation and write $B_x(n)$ for the corresponding indecomposable modules in $\SMod(\h)$. 
	\\
	\par 
	The \ldef{mixed derived Hecke category} $D^{\mix}(\h):= K^b(\SMod(\h))$ is the bounded homotopy category of $\SMod(\h)$. 
	It is a triangulated category, with two shift-functors: $(1)$ inherited from $\SMod(\h)$, and $[1]$ from cohomological shift.
	We define $\tate{1} = (-1)[1]$. 
	By an abuse of notation, we denote the complex consisting exclusively of the Soergel module $B_x$ in cohomological degree $0$ by $B_x$. 
	In \cite{EW14,ARV20} it is shown that $D^{\mix}(\h)$ can be endowed with a canonical perverse $t$-structure.
	We do not require the precise definition of the $t$-structure. 
	\\
	\par 
	The \ldef{mixed perverse Hecke category} $\Perv(\h)$ is defined to be the heart of the perverse $t$-structure on $D^{\mix}(\h)$.
	It is a graded highest weight category, in the sense of \cite[Appendix A]{AR16b}, with shift-functor $\tate{1}$ inherited from $D^{\mix}(\h)$ \cite[\S9.5]{ARV20}.
	For any $A, A'$ in $\Perv(\h)$ we define 
	\begin{align*}
		\Hom_{\Perv(\h)}^{\bullet} (A,A') 
		:= 
		\bigoplus_{k \in \Z} \Hom_{\Perv(\h)}(A,A'\tate{k}). 
	\end{align*}
	For each $x \in W$ and $n \in \Z$ there are complexes $L_x \tate{n}, \Delta_x \tate{n}, \nabla_x \tate{n}, T_x\tate{n}$, which are simple, standard, costandard and tilting objects respectively. 
	These exhaust all isomorphism classes of objects of these types of objects.  
	Since we assume $W$ is finite, the category $\Perv(\h)$ has enough injectives and projectives. 
	Denote by $I_x$ (resp. $P_x$) an injective envelope (resp. projective cover) of $L_x$. 
	Note that in general the complex $B_x$ need not be perverse. 
	\\
	\par
	Fix a reflection faithful realisation $\h$ and construct the categories $\SBim(\h)$, $\SMod(\h)$, $D^{\mix}(\h)$, and $\Perv(\h)$.
	To utilise the powerful tool that is Koszul duality, we require a slight modification of the corresponding categories for the Langlands dual realisation $\check{\h}$. 
	The category $\SBim(\check{\h})$ is defined completely analogously to above. 
	The category $\SMod(\check{\h})$ is now the category of \textit{left} Soergel modules. 
	That is, $(-)\otimes_{\check{R}} \Bbbk$ induces a functor from $\check{R}$-$\Bim$ to $\check{R}$-$\Mod$, the category of left $\check{R}$-modules. 
	 Then $\SMod(\check{\h})$ is defined as the essential image of $\SBim(\check{\h})$ under this functor.
	 With this modification $D^{\mix}(\check{\h})$ and $\Perv(\check{\h})$ are defined completely analogously to above. 

\section{Functors on Hecke categories}
\label{Sec: functors}
	
	Mixed Hecke categories are endowed with various well-known dualities and functors. 
	We briefly recall these functors, as they will be utilised in Section \ref{Sec:JH mult}.  
	\\
	\par	
	Since we assume $W$ is finite with longest element $w_0$, the category $D^{\mix}(\h)$ admits an autoequivalence $\tR$, called  \ldef{Ringel duality}. 
	The construction of $\tR$ is not relevant to our purposes, see \cite[\S 10.1]{ARV20} for details. 
	Our interest lies in the following fact. 
	If we set $\Tilt(\h)$, $\Inj(\h)$ and $\Proj(\h)$ to be the full, additive subcategories of tilting, injective and projective objects respectively, then Ringel duality induces equivalences which satify:
	\begin{align*}
		\Inj(\h) \tilde{\longrightarrow} \Tilt(\h)~ \tilde{\longrightarrow} ~\Proj(\h)
		\\
		I_{xw_0}\tate{n} \longmapsto T_x \tate{n} \longmapsto P_{xw_0}\tate{n}
	\end{align*}
	for each $x \in W$, see \cite[\S 10.2]{ARV20}. 
	\\
	\par
	The second duality we exploit is the considerably deeper equivalence $\kappa$ called \ldef{Koszul duality}.
	The version we consider, is that constructed in \cite{AMRW19, RV23}. 
	Namely, we have an equivalence which satisfies (among other properties):
	\begin{align*}
		D^{\mix} (\h) &\tilde{\longrightarrow} D^{\mix} (\check{\h})
		\\
		T_x &\longmapsto \check{B}_x
	\end{align*}
	for every $x \in W$ and $n \in \Z$, and there is a natural isomorphism $\kappa \circ \tate{n} \cong (n) \circ \kappa$. 
	Note the presence of the Langlands dual realisation in Koszul duality, and the change in grading. 
	\\
	\par 
	Classically, Soergel's functor $\bV$ is defined as the functor induced by taking homomorphisms from the big projective in category $\mathcal{O}$, \cite{Soe90}.
	In $\Perv(\h)$ and $D^{\mix}(\h)$ we have the following isomorphisms
	\begin{align}
	\label{Eqn: tilt proj iso}
		I_{\id}\tate{-\ell(w_0)} \cong T_{w_0} \cong P_{\id} \tate{\ell(w_0)}
	\end{align}
	due to \cite[\S10.3]{ARV20}. 
	This suggests the following graded analogue of \ldef{Soergel's functor}
	\begin{align*}
		\bV^{\bullet}:= \Hom_{D^{\mix}(\h)}^{\bullet}(T_{w_0}\tate{-\ell(w_0)},-)
	\end{align*}
	If $\h$ is a root realisation we have, by \cite[\S3.6]{AMRW19}, a graded analogue of Soergel's Struktursatz. 
	Namely, $\bV^{\bullet}$ induces an equivalence of additive categories
	\begin{align*}
		\Tilt(\h) \tilde{\longrightarrow} \SMod(\check{\h})
	\end{align*}
	where, for each $x \in W$, we have $\bV^{\bullet}(T_{x}) \cong \check{B}_x$.
	
\section{Graded ranks of Soergel modules}
\label{Sec: ranks}

	We now fix notation for graded modules, and express the graded ranks of Soergel (bi)modules in terms of the graded multiplicities of standard bimodules in the bimodule $B_x$. 
	\\
	\par 
	Given a commutative ring $A$, and a graded, free $A$-module $M \cong \bigoplus_{i \in \Z} M^i$ we define the graded rank of $M$, as an $A$-module, as
	\begin{align*}
		\text{rk}_A^{\bullet} M 
		= 
		\sum_{i \in \Z} \text{rk}_A M^i ~v^i 
		\in 
		\Z[v,v^{-1}].
	\end{align*}
	When $A$ is clear from context we omit it from notation, simply writing $\grrk M$ instead. 
	\\
	\par
	For any $x \in W$, define the \ldef{standard bimodule} $R_x$ in $R$-$\Bim$ to be the $R$-bimodule where $R_x \cong R$ as left $R$-modules and the right action is given by $m \cdot r = m(x(r))$ for any $m \in R_x$ and $r \in R$, i.e. the right action is twisted by $x$. 
	Fix an enumeration $x_0 , \dots, x_{k}$ of elements in $W$ which refines the Bruhat order, i.e. $x_i < x_j$ implies $i<j$.
	An $R$-bimodule $B$ is said to have a \ldef{standard filtration}, relative to the enumeration, if it has a filtration $0 = B_j \subset \dots \subset  B_0 = B$ satisfying $B_i/B_{i-1} \cong \bigoplus_{n \in \Z} R_{x_i}(n)^{\oplus m_{i,n}}$.
	Given an $R$-bimodule $B$ with standard filtration we write $h_{x_i}(B)= \sum_{n \in \Z} m_{i,n} v^n $ for the graded multiplicities of $R_{x_i}$ in this filtration.   
	\\
	\par 
	It was shown in \cite{Soe07} that after fixing a choice of enumeration any Soergel bimodule $B$ admits a unique standard filtration, and the graded multiplicity $h_y(B)$ is independent of the choice of enumeration. 
	Consequently for any Soergel bimodule $B$ we have
	\begin{align*}
		\grrk B = \sum_{y \in W} h_y(B). 
	\end{align*}
	It follow that the graded rank of the Soergel  module $\Bbbk \otimes_R B$ is  $\text{rk}^{\bullet}_{\Bbbk}   (\Bbbk \otimes_R B) = \text{rk}^{\bullet}_R B = \sum_y h_y (B)$.  
	\\
	\par 
	For each of the realisations we consider, the category of Soergel (bi)modules satisfies Soergel's conjecture, see \cite{Soe90, EW14}. This is equivalent to the statement that for all $x,y \in W$ we have $h_y(B_x) = v^{-\ell(y)} h_{y,x} $, where $h_{y,x}$ is the classical Kazhdan-Lusztig polynomial (in the normalisation of \cite{Soe97}). 
	Thus, we have
	\begin{align*}
		\grrk B_x = \grrk \check{B}_x = \sum_y  v^{-\ell(y)} h_{y,x}. 
	\end{align*} 
	
	\begin{rem}
		More generally, for any reflection faithful realisation $\h$ the multiplicity $h_y(B_x)$ will be $v^{-\ell(y)} ~{}^ph_{y,x}$, where ${}^p h_{y,x}$ is the $p$-Kazhdan-Lusztig polynomial introduced in \cite{JW17}. 
	\end{rem}
	
	\par
	We conclude by observing that $\grrk B_x$ satisfies a parity property. 
	\begin{lem}
	\label{Lem: parity}
		For any $x \in W$ we have $\grrk B_x \in v^{\ell(x)}\Z[v^{-2}]$.
	\end{lem}
	\begin{proof}
		We prove the claim for Soergel bimodules. 
		For any $s \in S$  we have an isomorphism of left $R$-modules $B_s \cong R(-1) \oplus R(1)$, see \cite[\S4]{Soe07}, so $\grrk B_s = v+v^{-1}$.
		Consequently, if $s_1 \dots s_k$ is an expression, then the Bott-Samelson bimodule $BS = B_{s_1} \otimes_R \dots \otimes_R B_{s_k}$ has graded rank $\grrk BS = (v+v^{-1})^{k} \in v^k \Z[v^{-2}]$.  
		The claim then follows from the facts that: (1) $\SBim(\h)$ is a Krull-Schmidt category; and, (2) if $s_1 \dots s_k$ is a reduced expression for $x$ then $B_x$ is a direct summand of $BS$ with multiplicity 1.
	\end{proof}

\section{Jordan-H\"{o}lder multiplicities of the big tilting object}
\label{Sec:JH mult}

	Recall $W$ is finite with longest element $w_0$. 
	In this section we determine the Jordan-H\"{o}lder multiplicities of $T_{w_0}$ and deduce a formula for $[T_{w_0}]$ in the split Grothendieck ring of $[\Perv(\h)]$. 
	
	\begin{lem}
	\label{Lem: grrk Soergel functor}
		For any $x \in W$ we have $\grrk \bV^{\bullet}(T_{x}) = \grrk \check{B}_{x}$. 
	\end{lem}
	
	We provide two proofs of the Lemma. The first is morally correct and justifies the presence of $\check{B}_{x}$ beyond numerical serendipity. The second holds more generally.  
	
	\begin{proof}[Proof (specific to root realisations)]
		When $\h$ is a root realisation, the claim is immediate from the graded analogue of Soergel's Struktursatz, i.e. the isomorphism $\bV^{\bullet}(T_{x}) \cong  \check{B}_{x}$. 
		To be more precise, in \cite[Lemmas 3.9 and 3.10]{AMRW19} it is shown that a `left-monodromic' analogue of $\bV^{\bullet}$ is fully faithful from the category of left-monodromic tilting complexes into the category $\SMod(\check{\h})$. 
		By \cite[Propositions 2.1]{AMRW19} the category of left-monodromic tilting complexes is equivalent to $\Tilt(\h)$, so the composition of these equivalences proves the claim.  
	\end{proof}
		
	\begin{proof}[Proof (that holds in the absence of a graded Struktursatz)]
		Let $\h$ be the geometric realisation of $(W,S)$. 
		Then Koszul duality implies 
		\begin{align*}
		 	\grrk \bV^{\bullet}(T_{x}) 
		 	&=
		 	\grrk \Hom^{\bullet}_{D^{\mix}(\h)} (T_{w_0} \tate{-\ell(w_0)} , T_{x})
		 	\\
		 	&=
		 	\grrk \Hom^{\bullet}_{D^{\mix}(\check{\h})} (\check{B}_{w_0} (-\ell(w_0)) , \check{B}_{x})
		 	\\
		 	&=
		 	\grrk \Hom^{\bullet}_{\SMod(\check{\h})} (\check{B}_{w_0} (-\ell(w_0)) , \check{B}_{x}). 
		\end{align*}
		Soergel's Hom formula \cite[\S 5]{Soe07} and its analogue for Soergel modules \cite[\S 1.7]{Ric19} implies  
		\begin{align*}
			\grrk \Hom^{\bullet}_{\SMod(\check{\h})} (\check{B}_{w_0} (-\ell(w_0)) , \check{B}_{x})
			&=
			\sum v^{2\ell(y)-\ell(w_0)}h_y(\check{B}_{w_0})~ h_y(\check{B}_x).
		\end{align*}
		For any finite Coxeter group $W$ we have an explicit description for the Soergel module $\check{B}_{w_0}$, namely $\check{B}_{w_0} \cong \check{R} \otimes_{\check{R}^W} \Bbbk (\ell(w_0))$, i.e. $\check{B}_{w_0}$ is isomorphic to the coinvariant algebra (normalised so that it degree-symmetric about zero). 
		Consequently $h_y(\check{B}_{w_0}) = v^{\ell(w_0) - 2\ell(y)}$. 
		Hence 
		\begin{align*}
			\grrk \Hom^{\bullet}_{\SBim(\check{\h})} (\check{B}_{w_0} (-\ell(w_0)) , \check{B}_{x})
			= 
			\sum h_y(\check{B}_x)
			= 
			\grrk \check{B}_x
		\end{align*}
		which completes the proof. 
	\end{proof}

	We now determine the graded Jordan-H\"{o}lder multiplicities of the big tilting object, $T_{w_0}$.
	
	\begin{lem}
	\label{Lem: JH is grrk}
		For any $x \in W$ we have 
		\begin{align*}
			\sum_{i \in \Z} [T_{w_0} : L_x \tate{i}] v^{i} 
			=
			\grrk \check{B}_{xw_0}
		\end{align*}
	\end{lem}
	\begin{proof}
		Since $I_x\tate{i}$ is an injective envelope of $L_x \tate{i}$, for any object in $A$ in $\Perv(\h)$, we have  
		\begin{align*}
			\sum_{i \in \Z} [A : L_x \tate{i}] v^{i}
			= 
			\sum_{i \in \Z} \text{rk} \Hom_{\Perv(\h)} (A , I_x\tate{i}) v^{i} = \grrk \Hom^{\bullet}_{\Perv(\h)}(A, I_x).
		\end{align*}
		Now observe that, by the isomorphism in Equation (\ref{Eqn: tilt proj iso}) and Ringel duality, we have 
		\begin{align*}
			\Hom^{\bullet}_{\Perv(\h)}(T_{w_0}, I_x) 
			&\cong 
			\Hom^{\bullet}_{\Perv(\h)}(I_{\id}\tate{-\ell(w_0)}, I_x) 
			\\
			&\cong 
			\Hom^{\bullet}_{D^{\mix}(\h)}(I_{\id}\tate{-\ell(w_0)}, I_x) 
			\\
			&\cong
			\Hom^{\bullet}_{D^{\mix}(\h)}(T_{w_0}\tate{-\ell(w_0)}, T_{xw_0}) 
			\\
			&=
			\bV^{\bullet}(T_{xw_0}).
		\end{align*}
		The claim then follows from Lemma \ref{Lem: grrk Soergel functor}. 
	\end{proof}

	Before we prove the main result of this section, we need to briefly discuss Grothendieck groups of categories with shift. 
	\\
	\par
	Let $[\SMod(\h)]$ denote the split Grothendieck group of $\SMod(\h)$. 
	The shift autoequivalence $(1)$ allows us to endow $[\SMod(\h)]$ with the structure of a $\Z[v,v^{-1}]$-module by imposing the relation $[B(1)] = v[B]$ for any object $B$. 
	Let $[D^{\mix}(\h)]$ denote the triangulated Grothendieck group of $D^{\mix}(\h)$.
	The $\Z[v,v^{-1}]$-module structure on $[\SMod(\h)]$ induces a $\Z[v,v^{-1}]$-module structure on $[D^{\mix}(\h)]$. 
	However, the category $D^{\mix}(\h)$ has an additional autoequivalence $[1]$ coming from cohomological shift. 
	For any complex $C$ the triangulated structure requires we impose the relation $[C[1]] = -[C]$.  
	Since $\Perv(\h)$ is stable under $\tate{1}= [1](-1)$, we endow $[\Perv(\h)]$ with the structure of a $\Z[v,v^{-1}]$-module, where  $[A \tate{1}] = -v^{-1}[A]$ for any object $A$. 
	\\
	\par 
	We can now state the main result of this section. 
	\begin{prop}
	\label{Prop: gro ring}
		In the split Grothendieck group $[\Perv(\h)]$, we have the following identity
		\begin{align*}
			[T_{w_0}] 
			= 
			\sum_{x} (-1)^{\ell(xw_0)} 
			\grrk \check{B}_{ x w_0}
			~ [L_x].
		\end{align*}
	\end{prop}

	\begin{proof}
		Fix $x \in W$, then Lemmas \ref{Lem: parity} and \ref{Lem: JH is grrk} imply
		\begin{align*}
			\sum_{i \in \Z} [T_{w_0} : L_x \tate{i}] [L_x \tate{i}]
			&=
			\sum_{i \in \Z} [T_{w_0} : L_x \tate{i}][L_x] (-v)^{-i}   
			\\&=
			 (-1)^{\ell(xw_0)} \sum_{i \in \Z} [T_{w_0} : L_x \tate{i}][L_x] v^{-i} 
			\\&= 
			(-1)^{\ell(xw_0)} \grrk \check{B}_{ x w_0} [L_x].
		\end{align*}
		Hence, by considering any composition series of $T_{w_0}$, we obtain
		\begin{align*}
			[T_{w_0}] 
			= 
			\sum_{x \in W, i \in \Z} [T_{w_0} : L_x \tate{i}] [L_x \tate{i}]
			= 
			\sum_{x \in W} (-1)^{\ell( x w_0)} \grrk \check{B}_{ x w_0}~ [L_x]
		\end{align*}
		which is the desired identity.
	\end{proof}
	
	\begin{rem}
		The statement of each result in this section remains valid if $\h$ is a realisation satisfying the assumptions of \cite[\S 2.1]{RV23}; this includes realisations defined over fields of positive characteristic. If $\h$ is not reflection faithful then $\check{B}_x$ should be interpreted as the relevant indecomposable Abe bimodule, as introduced in \cite{Abe21}.
	\end{rem}

\section{Hecke algebras}
\label{Sec: Hecke algebras}

	In this section we fix notation relating to Hecke algebras and deduce an expression for the antisymmetriser idempotent. 
	\\
	\par
	The \ldef{Hecke algebra} $H$ associated to a Coxeter system $(W,S)$ is the associative $\Z[v,v^{-1}]$-algebra on the symbols $\{ \delta_x \vert x \in W\}$ subject to the relations
	\begin{align*}
		(\delta_s+v)(\delta_s-v^{-1}) &=0 && \text{for all } s\in S, \text{ and,}
		\\
		\delta_x \delta_y &= \delta_{xy} && \text{whenever } \ell(x) + \ell(y) = \ell(xy).
	\end{align*}
	The symbols $\{ \delta_x \vert x \in W\}$ are a basis of $H$ called the standard basis \cite{Tits69}. 
	It also has a canonical basis $\{ b_x \vert x \in W \}$ called the Kazhdan-Lusztig basis, which we normalise as in \cite{Soe97}. 
	In this normalisation, one has $b_s = \delta_s + v$ for each $s \in S$.  
	\\
	\par
	Let us recall two involutions on the Hecke algebra. 
	The \ldef{Koszul involution} $\kappa$ is the unique $\Z$-linear involution satisfying $\kappa(v) = -v^{-1}$ and $\kappa(\delta_x) = \delta_x$ for each $x \in W$. 
	The \ldef{Kazhdan-Lusztig involution} is the unique $\Z$-linear involution satisfying $\overline{v} = v^{-1}$ and $\overline{\delta_x} = \delta_{x^{-1}}^{-1}$ for each $x \in W$. 
	By definition, the Kazhdan-Lusztig basis satisfies $\overline{b_x} = b_x$ for each $x \in W$. 
	\\
	\par 
	Soergel's categorification theorem,  \cite[\S1.1]{Soe90} and \cite[\S6.6]{ARV20}, states that for the realisations we consider (and many more), there is a unique isomorphism of right $H$-modules  
	\begin{align*}
		[\SMod(\h)] \tilde{\longrightarrow} H
	\end{align*}
	which is induced by the map $[B_s] \mapsto b_s$ for each $s \in S$. 
	This extends to an isomorphism
	\begin{align*}
		[D^{\mix}(\h)] \tilde{\longrightarrow} H
	\end{align*} 
	which satisfies $[\Delta_x] = \delta_x$ for each $x \in W$. 
	\\
	\par 
	A realisation is said to satisfy \ldef{Soergel's conjecture} if $[B_x] = b_x$ for each $x \in W$. 
	Soergel showed, through recourse to geometry, that the root realisation (extended to a field of characteristic 0) satisfies Soergel's conjecture \cite{Soe92}. 
	It is a consequence of a celebrated theorem of Elias and Williamson \cite{EW14} that the geometric realisation also satisfies Soergel's conjecture. 
	\\
	\par 
	Henceforth, we write $H$ for the $\Q(v)$-algebra $H \otimes_{\Z[v,v^{-1}]} \Q(v)$. 
	\\
	\par 
	Each Hecke algebra has a unique rank 1 $H$-module where $\delta_x$ acts by $(-v)^{\ell(x)}$ for each $x \in W$, called the (quantised) sign module.
	The \ldef{antisymmetriser idempotent} $e_{\sign}$ is a primitive idempotent in $H$ satisfying $e_{\sign}H$ is isomorphic to the quantised sign module.
	Since $W$ is finite, $e_{\sign}$ exists and is unique up to sign.
	It is well-known and easy to check 
	\begin{align*}
		 e_{\sign} = \left( \prod_{k=1}^n\frac{v-v^{-1}}{v^k - v^{-k}}  \right) \sum_{x \in W} (-1)^{\ell(x)} v^{-\ell( x w_0)} \delta_x.
	\end{align*}
	We normalise $e_{\sign}$ so that, when expressed in the Kazhdan-Lusztig basis, the coefficient of $b_{\id}$ is positive.  
	
	\begin{prop}
	\label{Prop: antisymmetriser}
		An expression for the antisymmetriser idempotent is
		\begin{align*}
			e_{\sign} 
			= 
			\sum_{x \in W} 
			(-1)^{\ell(x)}
			~
			\frac{\grrk \check{B}_{ x w_0}}{\grrk \check{B}_{w_0}}
			~
			b_x.
		\end{align*} 
	\end{prop}
	\begin{proof}
		Being a tilting object, $T_{w_0}$ admits a graded $\Delta$-filtration, i.e. a filtration where successive quotients are of the form $\Delta_x \tate{n}$ for some $x \in W$ and $n \in \Z$. 
		The graded $\Delta$-filtration multiplicities of $T_{w_0}$, $(T_{w_0} : \Delta_x \tate{n})$, are known to be 
		\begin{align*}
			(T_{w_0} : \Delta_x \tate{n}) 
			= 
			\begin{cases}
				1
				& \text{if } n = \ell( x w_0),
			\\
				0
				& \text{otherwise}
			\end{cases}
		\end{align*}
		by \cite[\S 10.3]{ARV20}. Note that by Koszul duality, this is equivalent to the well-known statement that the Kazhdan-Lusztig polynomial $h_{x,w_0}$ is equal to $v^{\ell( x w_0)}$.
		Thus, after identifying $[D^{\mix}(\h)]$ with $H$, we have 
		\begin{align*}
			[T_{w_0}] 
			= \sum_{x \in W}  [\Delta_x \tate{\ell( x w_0)}] 
			= \sum_{x \in W} (-v)^{-\ell(x w_0)} \delta_x
			= (-1)^{\ell(w_0)} \sum_{x \in W} (-1)^{\ell(x)} v^{-\ell( x w_0)} \delta_x. 
		\end{align*}
		This implies that $[T_{w_0}]$ is an antisymmetriser quasi-idempotent for any realisation $\h$, which need not satisfy Soergel's conjecture. 
		Hence
		\begin{align*}
			[T_{w_0}]^2 
			= 
			\sum_{x \in W} v^{-\ell( x w_0) + \ell(x)} [T_{w_0}]
			=
			\grrk \check{B}_{w_0} [T_{w_0}],
		\end{align*}
		where the second equality follows from the fact $h_x(\check{B}_{w_0}) = v^{\ell(w_0) - 2\ell(x)}$. 
		Thus we obtain
		\begin{align*}
			e_{\sign} = \frac{(-1)^{\ell(w_0)} [T_{w_0}]}{\grrk \check{B}_{w_0}}
		\end{align*}
		where the factor $(-1)^{\ell(w_0)}$ will ensure that we agree with our sign convention for $e_{\sign}$.
		\\
		\par 
		It is known that if $\h$ satisfies Soergel's conjecture then $B_x \cong L_x$ in $\Perv(\h)$, \cite[\S 8.6]{ARV20}. 
		Hence after identifying $[\Perv(\h)]$ with $H$, one has $[L_x] = [B_x] =b_x$. 
		The claim is then immediate from Proposition \ref{Prop: gro ring} and the fact that all realisations we consider satisfy Soergel's conjecture. 
	\end{proof}
	
	\begin{rem}
		The structure constants $\mu_{y,x}^s$ appearing in $b_x b_s = \sum_z \mu_{y,x}^s b_y$ are, in general, extremely poorly understood and intimately related to $\mu(y,x)$, the coefficient of $v$ in $h_{y,x}$. 
		Since $[T_{w_0}] \delta_s = -v[T_{w_0}]$ and $b_s = \delta_s + v$, it follows $[T_{w_0}]b_s = 0$  for each $s \in S$. 
		This implies that for each $y \in W$ and $s \in S$ we have the following identity
		\begin{align*}
			\sum_x (-1)^{\ell(x)} ~\grrk \check{B}_{ x w_0}~ \mu^s_{y,x}=0,
		\end{align*} 
		which may be of independent interest. 
	\end{rem}
	
\section{Temperley-Lieb algebras and Jones-Wenzl idempotents}
\label{Sec: TL}

	We now introduce the Temperley-Lieb algebra and finally deduce the formula for the Jones-Wenzl idempotent that was stated in Theorem \ref{Thm: JW type A}. 
	\\
	\par 
	The \ldef{Temperley-Lieb algebra} $TL_n$ is the associative, unital $\Q(v)$-algebra generated by $u_1 , \dots, u_{n-1}$ subject to the relations:
	\begin{align*}
		u_i u_j &= u_j u_i
		&&
		\text{whenever } |i-j| > 1,
		\\
		u_i u_j u_i &= u_i
		&&
		\text{whenever } |i-j| = 1,
		\\
		u_i ^2 &= (v+v^{-1}) u_i
		&&
		\text{for each } 1 \leq i < n.
	\end{align*}
	There is also the algebra $TL_n^-$ defined analogously, where the final relation is replaced by the condition $u_i^2 = -(v+v^{-1})u_i$.  
	\\
	\par
	The Temperley-Lieb algebras are well-known to arise naturally as quotients of the Hecke algebra of type $\tA_{n-1}$. 
	In particular, let $(W,S)$ be a Coxeter system of type $\tA_{n-1}$, and let $s,t \in S$ satisfy $sts=tst$ and $s \neq t$, then we have a commutative diagram 
	\[
			\begin{tikzcd}
				0
				\arrow[r] 
			&
				\langle b_{sts} \rangle
				\arrow[r] 
				\arrow[d] 
			&  
				H
				\arrow[r, "\pi"] 
				\arrow[d] 
			&
				TL_{n}
				\arrow[r]
				\arrow[d]  
			&
				0
			\\
				0
				\arrow[r] 
			&
				\langle \kappa(b_{sts}) \rangle
				\arrow[r]  
			&  
				H
				\arrow[r , "\pi_-"] 
			&
				TL^-_{n}
				\arrow[r] 
			&
				0
			\end{tikzcd}
	\]
	where each row is exact, and each vertical arrow is induced by the Koszul involution $\kappa$.
	We emphasise that we write $\pi : H \rightarrow TL_n$ for the quotient map. 
	\\
	\par 
	Recall that an expression $s_1 \dots s_k$ is a reduced expression for $x$ if $x = s_1 \dots s_k$ and $\ell(x)=k$.
	Following \cite{Ste96}, an element $x$ is said to be \ldef{fully commutative} if no reduced experssion contains a substring which is a reduced expression for the longest element of a Coxeter system of type $\tI_2(m)$, where $m\geq 3$. 
	Equivalently, all reduced expressions for $x$ may be obtained from a single reduced expression by applying relations of the form $st=st$.
	We let $W_{FC} \subset W$ denote the subset of fully commutative elements. 
  	\\
	\par 
	Now impose the standard type $\tA_{n-1}$ ordering on the simple reflection $s \in S$. 
	Given a reduced expression $s_{i_1} \dots s_{i_k}$ for $x$, we define $u_x := u_{i_1} \dots u_{i_k}$. 
	The element $u_x$ is independent of choice of reduced expression. 
	Moreover, the set $\{ u_x \vert x \in W_{FC}\}$ is a $\Z[v,v^{-1}]$-basis of $TL_n$ called the \ldef{monomial basis} of $TL_n$, see \cite[\S2.2]{Fan96}. 
	It is easy to check that $\pi(b_{s_i}) = u_i$. 
	\\
	\par 
	The \ldef{Jones-Wenzl idempotent} $j_n$ in $TL_n$ is the unique idempotent satisfying $j_n u_i =0$ for all $1 \leq i < n$, and whose coefficient of the identity is 1 when expressed in the monomial basis. 
	\\
	\par
	We can now prove the main theorem. 
	\begin{proof}[Proof of Theorem \ref{Thm: JW type A}]
		By definition $e_{\sign}\delta_s  = -ve_{\sign}$ for all $s \in S$. Since $b_s = \delta_s + v$ and $\pi(b_s) = u_i$, it follows that $\pi(e_{\sign}) = j_n$. 
		It is a theorem of Fan and Green, see \cite[\S3.8]{FG97}, that 
		\begin{align*}
			\pi(b_x) = 
			\begin{cases}
				u_x & \text{if } x \in W_{FC}
				\\
				0 & \text{otherwise.}
			\end{cases}
		\end{align*}
		Applying this to Proposition \ref{Prop: antisymmetriser}, one finds
		\begin{align*}
			j_n 
			=
			\pi(e_{\sign})
			= 
			\sum_{x \in W} 
			(-1)^{\ell(x)}
			~
			\frac{\grrk \check{B}_{ x w_0}}{\grrk \check{B}_{w_0}}
			~
			\pi(b_x)
			=
			\sum_{x \in W_{FC}} 
			(-1)^{\ell(x)}
			~
			\frac{\grrk \check{B}_{ x w_0}}{\grrk \check{B}_{w_0}}
			~
			u_x.
		\end{align*}
		which proves the claim. 
	\end{proof}

	The following is a particularly lovely example of Theorem \ref{Thm: JW type A}. 
	\begin{exmp}
		Let $n=3$. 
		Recall the quantum integers $[2] = v+v^{-1}$ and $[3] = v^{-2} + 1 + v^{-2}$, and the quantum factorial $[3]! = [3][2]$.
		It is well-known that $j_3$ can be written as 
		\begin{align*}
			j_3 
			= 
			\frac{[3][2]}{[3]!} u_{\id} - \frac{[2][2]}{[3]!}u_1 - \frac{[2][2]}{[3]!}u_2 + \frac{[2]}{[3]!} u_{12} + \frac{[2]}{[3]!} u_{21}. 
		\end{align*}
		For $x \in W$, we define $[x] = \sum_{y \leq x} v^{\ell(x)-2\ell(y)}$, which is the Poincar\'{e} polynomial of the Bruhat interval $[\id, x]$. 
		As every Schubert variety in the flag variety $SL_3/B$ is smooth, we have $\grrk \check{B}_x = [x]$ for all $x \in W$. 
		Thus, in this special case, the coefficients are simply the Poincar\'{e} polynomials of various Bruhat intervals. 
		Namely
		\begin{align*}
			j_3 
			= 
			\frac{[s_1s_2s_1]}{[s_1s_2s_1]} u_{\id} - \frac{[s_2s_1]}{[s_1s_2s_1]}u_1 - \frac{[s_1s_2]}{[s_1s_2s_1]}u_2 + \frac{[s_2]}{[s_1s_2s_1]} u_{12} + \frac{[s_1]}{[s_1s_2s_1]} u_{21}. 
		\end{align*}
	\end{exmp}

	\begin{rem}
		Using the Koszul involution $\kappa$, and the parity property in Lemma \ref{Lem: parity}, one finds that the analogous formula for $j_n^-$ in $TL_n^-$ is:
		\begin{align*}
			j_n^- 
			= 
			\sum_{x \in W_{FC}} 
			~
			\frac{\grrk \check{B}_{ x w_0}}{\grrk \check{B}_{w_0}}
			~
			u_x^-
		\end{align*}
		where $\{ u_x^- ~\vert~ x \in W_{FC} \}$ is the monomial basis of $TL_n^-$. 
	\end{rem}
	
\section{Generalised Jones-Wenzl idempotents}
\label{Sec: generalised}

	We now present an analogue of Theorem \ref{Thm: JW type A} for generalised Jones-Wenzl idempotents.  
	\\
	\par 
	For any Coxeter system $(W,S)$, the generalised Temperley-Lieb algebra $TL_W$ was independently introduced in \cite{Gra96} and \cite{Fan96}. 
	Let $J < H$ be the ideal generated by 
	\begin{align*}
		\{ b_{w_I}~|~ I \subset S, ~ |I|=2, \text{ and }(W_I, I) \text{ is a Coxeter system of type } \tI_2(m) \text{ where } 2< m < \infty \}
	\end{align*}
	where $w_I$ denotes the longest element in $W_I$.
	The \ldef{generalised Temperley-Lieb algebras} $TL_W$ and $TL_W^-$ are defined by the commuting diagram 
	\[
			\begin{tikzcd}
				0
				\arrow[r] 
			&
				J
				\arrow[r] 
				\arrow[d] 
			&  
				H
				\arrow[r, "\pi"] 
				\arrow[d] 
			&
				TL_{W}
				\arrow[r]
				\arrow[d]  
			&
				0
			\\
				0
				\arrow[r] 
			&
				\kappa(J)
				\arrow[r]  
			&  
				H
				\arrow[r , "\pi_-"] 
			&
				TL^-_{W}
				\arrow[r] 
			&
				0
			\end{tikzcd}
	\]
	where each row is exact, and each vertical map is induced by the Koszul involution $\kappa$. 
	The \ldef{generalised Jones-Wenzl idempotent} $j_W$ is defined in \cite{Sen19} as $j_W = \pi(e_{\sign})$.
	\\
	\par 
	As in the classical case, if $s_1 \dots s_k$ is a reduced expression for $x$, then we define $u_x := \pi(b_{s_1} \dots b_{s_k})$. 
	If $x \in W_{FC}$, the element $u_x$ is independent of the choice of reduced expression, and $\{ u_x \vert x \in W_{FC} \}$ is the \ldef{monomial basis} of $TL_W$ \cite[\S3]{GL99}. 
	\\
	\par 
	The Kazhdan-Lusztig involution fixes $J$, so induces an involution on $TL_W$.
	In \cite{GL99}, the authors show that $TL_W$ admits an \ldef{IC-basis} $\{ \beta_x \vert x \in W_{FC} \}$, in the sense of \cite{Du94}, with respect to the induced involution.
	The IC and monomial bases of $TL_W$ only coincide in types $\tA$, $\tD$ and $\tE$, see \cite[\S3]{GL99}. 
	\\
	\par 
	The relationship between the Kazhdan-Lusztig basis of $H$ and the IC-basis of $TL_W$ is subtle. 
	A Coxeter system is said to have the \ldef{projection property} if $\pi(b_x) = \beta_x$ for each $x \in W_{FC}$.
	The projection property is known to hold in finite types (with the exception of types $\tE_6$, $\tE_7$, and $\tE_8$  which remain open), see \cite{GL99,GL01}, and no Coxeter systems are known where the projection property does not hold. 
	Furthermore, in types $\tA, \tB, \tC, \tF_4, \tH_3, \tH_4$, and $\tI_2(m)$ it is known $\pi(b_x) =0$ if $x \notin W_{FC}$ while in types $\tE_6, \tE_7, \tE_8$ and $\tD_n$, with $n \geq 4$, there are $x \in W \backslash W_{FC}$ where $\pi(b_x)\neq 0$ \cite{Los00,GL01,Gre07}. 
	\\
	\par 
	From the preceding paragraph, it is clear that with minor modification of the proof of Theorem \ref{Thm: JW type A}, we have:

	\begin{thm}
	\label{Thm: gen JW}
		Let $(W, S)$ be of type $\tA, \tB, \tC, \tF_4, \tH_3, \tH_4$, or $\tI_2(m)$. 
		The generalised Jones-Wenzl idempotent $j_W$ has the following form when expressed in terms of the IC-basis of $TL_W$:
		\begin{align*}
			j_W = 
			\sum_{x \in W_{FC}} 
			(-1)^{\ell(x)}
			~\frac{\grrk \check{B}_{ x w_0}}{\grrk \check{B}_{w_0}}
			~\beta_x.
		\end{align*} 
	\end{thm}
	
	\begin{rem}
		The main result of \cite[\S5]{Sen19} is a description of the coefficient of the basis element indexed by $x=  w_I w_0$, for some cominuscule pair $(W,S,I)$, when $j_W$ is expressed in terms of the basis $\{ \pi(\delta_x) \vert x \in W_{FC}\}$. 
		Theorem \ref{Thm: gen JW} implies that these are the easiest coefficients to determine in the IC-basis, as the Schubert variety indexed by $ x w_0= w_I$ is smooth (it is a (co)minuscule flag variety), and the polynomial $\grrk \check{B}_{w_0} / \grrk \check{B}_{w_I}$ is the Poincar\'{e} polynomial of the flag variety. 
		It is noteworthy that this agrees with the coefficient determined in \cite[\S5]{Sen19}. 
	\end{rem}

\section*{Acknowledgements}

	This note is the proof of a result contained in my PhD thesis which was completed at the University of Sydney under the supervision of Geordie Williamson.
	I would like to thank Geordie Williamson for introducing me to the world of Kazhdan-Lusztig theory and Hecke categories, and providing many valuable comments on preliminary versions of this paper, and Gus Lehrer for originally introducing me to Temperley-Lieb algebras and Jones-Wenzl idempotents.  
	The author was supported by the award of a Research Training Program scholarship.

%==========================================================================

%============================== BIBLIOGRAPHY ==============================

%==========================================================================
%==========================================================================

\end{document}